\documentclass[12pt,a4paper]{article}
\usepackage[T1]{fontenc}
\usepackage[british]{babel}
\usepackage[utf8]{inputenc}
\usepackage{amsmath,amssymb,amsthm}

\DeclareMathOperator{\Norm}{N}

\DeclareMathOperator{\C}{C}

\DeclareMathOperator{\Max}{Max}
\DeclareMathOperator{\F}{F}
\DeclareMathOperator{\Irr}{Irr}

\newtheorem{theorem}{Theorem}
\newtheorem{lemma}[theorem]{Lemma}

\newtheorem{athm}{Theorem}

\theoremstyle{definition}

\newtheorem{remark}[theorem]{Remark}

\title{Weak second maximal subgroups in solvable groups\thanks{The research of the work was partially supported by the National Natural Science Foundation of China(11771271).}}
\author{Hangyang Meng and Xiuyun Guo\thanks{Corresponding author. E-mail: xyguo@staff.shu.edu.cn}\\
Department of Mathematics, Shanghai University\\
Shanghai 200444, P. R. China}

\begin{document}
\date{}
\maketitle
\begin{abstract}
  In this paper, we investigate the differences between weak second maximal subgroups and second maximal subgroups. A sufficient and necessary condition is also given to describe a class of groups whose weak second maximal subgroups coincide with its second maximal subgroups(called WSM-groups) under the solvable case. As an application, we will prove that every non-vanishing element of a solvable WSM-group lies in its Fitting subgroup.


\end{abstract}

\section{Introduction}
All groups considered in paper are finite.

Recall that an element $x$ of a group $G$ is said to be the \emph{non-vanishing} element of $G$ if $\chi(x)\neq 0$ for all $\chi \in \Irr(G)$, where $\Irr(G)$ is the set of all irreducible complex character of $G$. It is clear that every central element of a group is non-vanishing. However, as the authors point out in ~\cite{IsaacsNavarroWolf1999}, not only may a non-vanishing element of a group be noncentral, it can even fail to lie in an abelian normal subgroup of the group. For all that, I. Issacs, G. Navarro and T. Wolf prove that every non-vanishing element of odd order in a solvable group must always lie in a nilpotent normal subgroup of the group (see Theorem D in \cite{IsaacsNavarroWolf1999}). They also conjecture that every non-vanishing element of a solvable group $G$ is contained in $\F(G)$, the Fitting subgroup of $G$. In this paper we find a sub-class of solvable groups, which is called solvable \emph{WSM}-groups. We can prove that:

\begin{athm}\label{A}
Let $G$ be a solvable \emph{WSM}-group and $x$ is a non-vanishing element of $G$.  Then $x\in \F(G)$.
\end{athm}

Now we introduce \emph{WSM}-groups, which also has its independent meaning since \emph{WSM}-groups can be regard as a generalization of supersolvable groups. Let $H$ be a proper subgroup of a group $G$. We denote by $\Max(G,H)$ the set of all maximal subgroups of $G$ containing $H$. A proper subgroup $H$ of $G$ is called a \emph{second maximal subgroup} of $G$ if $H$ is a maximal subgroup of every member of $\Max(G,H)$, and we say $H$ is a \emph{weak second maximal subgroup} of $G$ if $H$ is a maximal subgroup of some member of $\Max(G,H)$.(See equivalent definitions in \cite{Flavell1995})

It is clear that a second maximal subgroup must be a weak second maximal subgroup. However, the converse is not true in general. For instance, $H=\langle(12)\rangle$ is a weak second maximal subgroup of $G=S_4$ but it is not a second maximal subgroup of $G$. It is natural to ask what will happen for a group $G$ if every weak second maximal subgroup of $G$ is a second maximal subgroup of $G$? For this purpose, we should investigate the differences between weak second maximal subgroups and second maximal subgroups.

\begin{athm}\label{B}
Let $G$ be a solvable group and $H$ be a weak second maximal subgroup of $G$. Then there exists at most one member $X$ of $\Max(G,H)$ such that $H$ is not maximal in $X$.
\end{athm}

Theorem~\ref{B} does not hold for the non-solvable case and the counterexample will be shown in Section $2$. For convenience, we say a group is a \emph{WSM}-group if its every weak second maximal subgroup must be a second maximal subgroup. Our next result will show a equivalent condition for solvable \emph{WSM}-groups so that we may use it to prove Theorem A. Our original motivation is the problem on chief factors of $\emph{WSM}$-groups proposed in \cite[Problem 19.54]{KhukhroMazurov2018}.

In order to state our theorem, we make the following non-standard definition about modules. Let $G$ be a group and $V$ be a $G$-module. We call an irreducible $G$-module $V$ \emph{strongly irreducible} if $G=1$ or $V$ is an irreducible $M$-module for every maximal subgroup $M$ of $G$.
A chief factor $H/K$ of $G$ is called \emph{non-Frattini} if $H/K\nleqslant \Phi(G/K)$.

\begin{athm}\label{C}
Let $G$ be a solvable group. Then the following statements are equivalent:

$(a)$ $G$ is a \emph{WSM}-group;

$(b)$ Every non-Frattini chief factor of $G$, as a $G$-module, is strongly irreducible.
\end{athm}

\section{The proof of Theorem B}

The following lemma is for the general case.

\begin{lemma}\label{key lemma}
Let $G$ be a group and let $H$ be a subgroup of $G$. If there exist $M, X \in \Max(G,H)$ such that $H$ is maximal in $M$ but not maximal in $X$, then $H_G=M_G$.
\end{lemma}

\begin{proof}
It is clear that we may assume $H_G=1$. If $M_G\neq 1$, then we may take a minimal normal subgroup $N$ of $G$ such that $N\leq M_G$.  Since $H_G=1$, we see $N\nleqslant H$ and therefore $NH=M$ by the maximality of $H$ in $M$. Thus $H\leq H(X\cap N)=X\cap M<M$. The maximality of $H$ in $M$ implies $H=X\cap M$ and so $X\cap N\leq H$. Thus $N\nleqslant X$ and  $G=NX$. Consider the natural group isomorphism $\varphi$ from $NX/N$ to $X/X\cap N$ defined by $\varphi(xN)=x(X\cap N)$ with $x \in X$.
Noticing that $HN/N=M/N$ is maximal in $G/N=XN/N$, we see $H/X\cap N=H(X\cap N)/X\cap N=\varphi(HN/N)$ is maximal in $\varphi(XN/N)=X/X\cap N$.
It follows that $H$ is maximal in $X$, a contradiction. Thus $M_G=1$ and the lemma is true.
\end{proof}

\begin{proof}[{Proof of Theorem~\ref{B}}]
It is clear that we may assume $H_G=1$ and that $H$ is maximal in $M$ for some member $M \in \Max(G,H)$. Assume that there are $X_1, X_2$ in $\Max(G,H)$ such that $H$ is not maximal in $X_i$ for $i=1 ,2$. Our aim is to prove that $X_1=X_2$. Since $H\leqslant M\cap X_i<M$ for $i=1 ,2$, we see $H=M\cap X_i$ by the maximality of $H$ in $M$.

It follows from Lemma~\ref{key lemma} that $M_G=1$ and $G$ is a solvable primitive group. By~\cite[Theorem~ A.15.2(1)]{DoerkHawkes1992}, there exists an unique minimal normal subgroup $S$ of $G$ such that $G=SM$ and $S\cap M=1$. If $S\nleqslant X_1$, then $G=X_1S$ and $X_1\cap S=1$ since $S$ is an abelian minimal normal subgroup of $G$.
Consider natural isomorphisms between $M$ and $MS/S$ and between $X_1S/S$ and $X_1$. Noticing that $HS/S$ is maximal in $MS/S=G/S=X_1S/S$, we see $H$ is maximal in $X_1$, in contradiction to the choice of $X_1$. Thus we may assume that $S\leqslant X_i$ for $i=1 ,2$, and therefore $X_i=X_i \cap SM=S(X_i\cap M)=SH$, which implies that $X_1=SH=X_2$, as desired.
\end{proof}

\begin{remark}
Theorem~\ref{B} does not hold if we remove the solvability of $G$. In fact, let $A=A_p, B=A_{p-1}$, the alternating groups of degree $p$ and $p-1$, where $p$ is a prime greater than $5$. Here $B$ can be viewed as a maximal subgroup of $A$.
Now set $G=A\times A$. Then $X_1=A\times B$ and $X_2=B\times A$ are maximal subgroups of $G$.  Also set the diagonal groups
$$M=\{(x,x)| x \in A\}~and~H=\{(x,x)| x \in B\}.$$
Since $A,B$ are non-abelian simple groups, it follows from Theorem~\cite[Theorem 1.9.14]{Huppert1967} that $M$ is maximal in $G$ and that $H$ is maximal in $B\times B=X_1\cap X_2< X_i$. It is easy to see that $X_i \in \Max(G,H)$ and $H$ is not maximal in $X_i$ for $i=1,2$. Also $H$ is maximal in $M$ since $|M:H|=|A:B|=p$. Thus $H$ is a weak second maximal subgroup of $G$ but not a second maximal subgroup of $G$.
\end{remark}

\section{The proof of Theorem C}

Recall that an irreducible $G$-module $V$ is \emph{strongly irreducible} if $G=1$ or $V$ is also an irreducible $M$-module for every maximal subgroup $M$ of $G$.
It is clear that $V$ is an irreducible (reducible) $G$-module if and only if $V$ is an irreducible (reducible) $G/N$-module whenever $N$ is a normal subgroup of $G$ contained in $\C_G(V)$. Now we begin with the following lemma.

\begin{lemma}\label{strongly irreducible}
Let $G$ be a group and $V$ be a $G$-module. Suppose that $N\lhd G$ and $N\leqslant \C_G(V)$. Then $V$ is a strongly irreducible $G$-module if and only if $V$ is a strongly irreducible $G/N$-module.
\end{lemma}

\begin{proof}
If $V$ is a strongly irreducible $G$-module, then it is clear that $V$ is an irreducible $G/N$-module. Now we may assume that $G/N\neq 1$ and so $G\neq 1$. If $M/N$ is a maximal subgroup of $G/N$, then $M$ is maximal in $G$ and therefore $V$ is an irreducible $M$-module. It follows immediately that $V$, as $M/N$-module, is irreducible. Hence $V$ is a strongly irreducible $G/N$-module.

Conversely, if $V$ is a strongly irreducible $G/N$-module, then it is clear that $V$ is an irreducible $G$-module.
Now we may assume that $G\neq 1$ and $M$ is a maximal subgroup of $G$. Then $MN=M$ or $MN=G$ by the maximality of $M$, which implies that $MN/N=M/N$ is maximal in $G/N$ or $MN/N=G/N$. Thus $V$ is an irreducible $MN/N$-module and immediately $V$ is an irreducible $M$-module. Hence $V$ is a strongly irreducible $G$-module.
\end{proof}

\begin{proof}[{Proof of Theorem~\ref{C}}]

Assume $G$ is a \emph{WSM}-group and that $N/L$ is a non-Frattini chief factor of $G$.
If $N = G$, then $N/L$ has prime order and is an irreducible $U$–module for all $U\leq G$; whence $N/L$ is strongly
irreducible. Now we assume that $N<G$. As $N/L$ is not in $\Phi(G/L)$, there exists a maximal subgroup $M/L$ of $G/L$ such that $N/L$ is not contained in $M/L$. By the maximality of $M$ in $G$ and the solvability of $G$; we have that $N/L$ is abelian, $NM=G$ and $N\cap M = L$.  Note that $M/L\neq 1$ as $N<G$. If there exists a maximal
subgroup $H/L$ of $M/L$ that does not act irreducibly on $N/L$; then there exists an $H$-invariant
subgroup $J/L$ of $N/L$ with $L < J < N$. Then $H < JH < NH<NM=G$ while $H$ is a maximal subgroup of a maximal
subgroup of $G$. This implies that $G$ is not a \emph{WSM}-group, contradicting the hypotheses. Hence Statement $(b)$ is proved.

Conversely, we show that Statement $(b)$ implies Statement $(a)$. Now assume that $H$ is a weak second maximal subgroup but not a second maximal subgroup of $G$.
Then there exist $M, X \in \Max(G,H)$ such that $H$ is maximal in $M$ and not maximal in $X$. Lemma~\ref{key lemma} implies that $H_G=M_G$, written by $U$.
Then $M/U$ is a core-free maximal subgroup of $G/U$. By~\cite[Theorem~A.15.2(1)]{DoerkHawkes1992}, there exists an unique minimal normal subgroup $V/U$ of $G/U$ such that $G=VM$ and $V\cap M=U$. Observe that $X/U$ is also a maximal subgroup of $G/U$. If $V\nleqslant X$, then $G=VX$ and $V\cap X=U$. Since $H/U$ is maximal in $M/U$, we have $HV/V$ is maximal in $MV/V=G/V=XV/V$. It follows that $H$ is maximal in $X$, a contradiction.
Thus we may assume that $V\leqslant X$, and so $X=V(X\cap M)=VH$. Since $V/U$ is a non-Frattini chief factor of $G$, by the hypothesis, $V/U$ is a strongly irreducible $G$-module. It follows from Lemma~\ref{strongly irreducible} that $V/U$ is a strongly irreducible $G/V$-module since $V\leqslant \C_G(V/U)$.

Observe that the action of $G/V$ on $V/U$ and the action of $M/U$ on $V/U$ are equivalent. Thus, by hypothesis, $V/U$ is a strongly irreducible $M/U$-module. The maximality of $H/U$ in $M/U$ implies that $V/U$ is an irreducible $H/U$-module. It follows that $H/U$ is maximal in $HV/U=X/U$ and therefore $H$ is maximal in $X$, in contradiction to the choice of $H$ and $X$.
Thus Statement $(a)$ holds and the theorem is proved.
\end{proof}

\begin{remark}
It is clear that every one-dimensional module must be strongly irreducible. Thus, by Theroem~\ref{C}, every supersolvable group is a \emph{WSM}-group. However, the converse is generally not true. For example, if $V$ is an elementary abelian $3$-group of order $9$ and $\alpha$ is a fixed-point-free automorphism of $V$ of order $8$, then $G=V \langle \alpha \rangle$ is not supersolvable but $G$ is a \emph{WSM}-group. In this viewpoint, solvable \emph{WSM}-groups can be regarded as a generalization of supersolvable groups.
\end{remark}

\section{The proof of Theorem A}

Recall that $V$ is a quasi-primitive $G$-module if $V_N$ is homogeneous for all normal subgroups $N$ of $G$.

\begin{lemma}\label{primitive}
Every strongly irreducible module is quasi-primitive.
\end{lemma}
\begin{proof}
Assume that a $G$-module $V$ is strongly irreducible but not quasi-primitive. Clearly $G\neq 1$. By Clifford’s Theorem, there is non-trivial decomposition  of $V$ into a direct sum of subspaces $V=V_1\oplus...\oplus V_n$ ($n>1$) such that $G$ permutes transitively on the set $\{V_1, \ldots V_n\}$.  Do such decomposition and make $n$ as small as possible. In this case, $\Norm_G(V_1)$ is a maximal subgroup of $G$ and $V$ is a reducible $\Norm_G(V_1)$-module. This implies that $V$ is not a strongly irreducible $G$-module, contradicting the hypotheses.
\end{proof}

The following result is the key to the proof of Theorem~\ref{A}.  We refer the
reader to the paper $\cite{Wolf2014}$ for more results about non-vanishing elements.

\begin{lemma}\emph{\cite[Theorem 2.1.]{Wolf2014}}\label{quasi-primitive}
Suppose that $V$ is a faithful quasi-primitive $G$-module and there exists $1\neq x \in \F(G)$
such that each element of $V$ is centralized by a $G$-conjugate of $x$. Then
\begin{itemize}
\item[$(a)$] $|V|=q^2$ for a Mersenne prime $q$ and $G=\F(G)=P\times S\subseteq \Gamma(V)$ where $S$ is cyclic of odd
order and $P$ is the Sylow-2-subgroup of $\Gamma(V)$ that is semi-dihedral of order $4(q + 1)$;
\item[$(b)$] $|V|=5^2$ and $\F(G)=QZ$ for normal subgroups $Q\cong Q_8$ and $Z\cong Z_4$ of $G$ with $Q\cap Z = Z(Q)$
while $G/\F(G)\cong Z_3$ or $S_3$; or
\item[$(c)$] $|V|=3^4$ and $\F(G)$ is isomorphic to a central product $Q_8 Y D_8$ and $G/\F(G)$ is isomorphic to
$Z_5, D_{10}$, the Frobenius group $F_{20}, A_5$, or $S_5$.
\end{itemize}
In all cases, $x$ is an involution and for each non-zero $v \in V$, $\C_{\F(G)}(v)=\langle x^g \rangle$ for some $g \in G$.
\end{lemma}

Let $V$ be a $G$-module. Recall the action of $G$ on its dual group $\Irr(V)$, the set of all complex characters of $V$. For any $\chi \in \Irr(V)$ and $g \in G$, define $\chi^g$ by
$$\chi^g(a)=\chi(a^{g^{-1}}), a \in V.$$

\begin{lemma}\label{the action on characters}
Let $G$ be a group and $V$ be a $G$-module. Then

$(a)$ $\C_G(V)=\C_G(\Irr(V))$.

$(b)$ If $V$ is an irreducible $G$-module, then $\Irr(V)$ is an irreducible $G$-module.

$(c)$ If $V$ is a faithful strongly irreducible $G$-module, then $\Irr(V)$ is a faithful strongly irreducible $G$-module.
\end{lemma}
\begin{proof}
$(a),(b)$ easily follow from the definition and~\cite[Proposition 12.1]{ManzWolf1993}. We only prove $(c)$. Suppose that $V$ is a faithful strong irreducible $G$-module and we may assume $G\neq 1$. We see that $(a)$ implies that $\Irr(V)$ is a faithful $G$-module. For every maximal subgroup $M$ of $G$, by hypothesis, $V$ is an irreducible $M$-module. It follows from $(b)$ that $\Irr(V)$ is an irreducible $M$-module. Hence $\Irr(V)$ is a strongly irreducible $G$-module, as desired.
\end{proof}

\begin{proof}[{Proof of Theorem~\ref{A}}]
We work by induction on $|G|$. Apply induction on $G/\Phi(G)$, we may assume $\Phi(G)=1$. Since $G$ is solvable, by~\cite[Theorem~13.8(b)]{DoerkHawkes1992}, we write $\F(G)=K_1\times...\times K_r$, where $K_i$ is a minimal normal subgroup of $G$ and $\F(G)=\cap_{i=1}^r C_i$, where $C_i=\C_G(K_i)$. It suffices to prove that $x \in C_i$ for each $i$.

Since $xC_i$ is a non-vanishing element of the\emph{WSM}-group $G/C_i$, by induction, we have $xC_i \in \F(G/C_i)$. Observe that $K_i$ is a non-Frattini chief factor of $G$.  By Theorem~\ref{C} and Lemma~\ref{strongly irreducible}, $K_i$ is a faithful strongly irreducible $G/C_i$-module.  Thus, by Lemma~\ref{the action on characters}$(c)$, $\Irr(K_i)$ is a faithful strongly irreducible $G/C_i$-module . It follows from Lemma~\ref{primitive} that $G/C_i$ acts faithfully and quasi-primitively on $\Irr(K_i)$. Since $x$ fixes some member of each $G$-orbit of $\Irr(K_i)$ by~\cite[Lemma~2.3]{IsaacsNavarroWolf1999},
we have that $xC_i$ fixes some member of each $G/C_i$-orbit in $\Irr(K_i)$, i.e. every element of $\Irr(K_i)$ is centralized by a $G/C_i$-conjugate of $xC_i$.  If $x \notin C_i$, applying Lemma~\ref{quasi-primitive}, $xC_i$ and $G/C_i$ satisfy the Conclusion $(a),(b)$ or $(c)$ of Lemma~\ref{quasi-primitive} and it is not difficult to see that $\F(G/C_i)$ has an irreducible character $\theta$ such that $\theta(xC_i)=0$ and $\theta$ is $G/C_i$-invariant, which implies that each $\beta \in \Irr(G/C_i|\theta)$ vanishes $xC_i$. Hence $xC_i$ is not a non-vanishing element of $G/C_i$. This is a contradiction and so $x \in C_i$, as desired.
\end{proof}

{\bf Acknowledgements.} The authors would like to thank the referee for pointing out a mistake in an original version and for suggesting us to shorten the proof of Theorem~\ref{A} significantly by using the results in the reference~\cite{Wolf2014}.

\bibliographystyle{plain}
  \bibliography{bibgroup}

\begin{thebibliography}{1}

\bibitem{DoerkHawkes1992}
K.~Doerk and T.~O. Hawkes.
\newblock {\em Finite soluble groups}, volume~4.
\newblock Walter de Gruyter, 1992.

\bibitem{Flavell1995}
P.~Flavell.
\newblock Overgroups of second maximal subgroups.
\newblock {\em Archiv der Mathematik}, 64(4):277--282, 1995.

\bibitem{Huppert1967}
B.~Huppert.
\newblock {\em {E}ndliche {G}ruppen {I}}, volume 134.
\newblock 1967.

\bibitem{IsaacsNavarroWolf1999}
I.~M. Isaacs, G.~Navarro, and T.~R. Wolf.
\newblock Finite group elements where no irreducible character vanishes.
\newblock {\em Journal of Algebra}, 222(2):413--423, 1999.

\bibitem{KhukhroMazurov2018}
E.~I. Khukhro and V.~D. Mazurov.
\newblock {\em Unsolved problems in group theory: The Kourovka Notebook}.
\newblock 2018.

\bibitem{ManzWolf1993}
O.~Manz and T.~R. Wolf.
\newblock {\em Representations of solvable groups}, volume 185 of {\em
  Mathematical Society Lecture Note Series}.
\newblock Cambridge University Press, London, 1993.

\bibitem{Wolf2014}
T.~R. Wolf.
\newblock Group actions related to non-vanishing elements.
\newblock {\em International Journal of Group Theory}, 3(2):41--51, 2014.

\end{thebibliography}
\end{document}